\documentclass[11pt]{amsart}

\usepackage{amsmath,amssymb,amsthm,graphics,amscd}
\usepackage{longtable}
\usepackage{cite,xypic}
\usepackage{color}
\usepackage{graphicx}
\usepackage{epsf}
\usepackage{fancyhdr,hyperref}
\usepackage{lscape}
\hypersetup{backref,
colorlinks=true,backref}

  \renewenvironment{thebibliography}[1]{
    \begin{oldthebibliography}{#1}
      \setlength{\parskip}{0ex}
      \setlength{\itemsep}{0ex}
  }
  {
    \end{oldthebibliography}
  }
  
\setlongtables

\DeclareMathOperator{\Hom}{Hom}
\DeclareMathOperator{\im}{im}
\DeclareMathOperator{\Ext}{Ext}

\DeclareMathOperator{\soc}{Soc}

\newcommand{\Z}{\mathbb{Z}}

\newcommand{\id}{\mathrm{id}}

\newcommand{\SL}{\mathrm{SL}}

\newcommand{\fit}{\phantom{\Big(}}

\begin{document}

\newcounter{rownum}
\setcounter{rownum}{0}

\newtheorem{lemma}{Lemma}[section]
\newtheorem{theorem}[lemma]{Theorem}
\newtheorem*{theorem*}{Theorem}
\newtheorem*{claim}{Claim}
\newtheorem{cor}[lemma]{Corollary}
\newtheorem{conjecture}[lemma]{Conjecture}
\newtheorem*{conjecture*}{Conjecture}
\newtheorem*{prop*}{Proposition}
\newtheorem{prop}[lemma]{Proposition}
\theoremstyle{remark}
\newtheorem{remark}[lemma]{Remark}
\newtheorem{obs}[lemma]{Observation}
\theoremstyle{definition}
\newtheorem{defn}[lemma]{Definition}

  \def\hal{\unskip\nobreak\hfil\penalty50\hskip10pt\hbox{}\nobreak
  \hfill\vrule height 5pt width 6pt depth 1pt\par\vskip 2mm}

\newenvironment{changemargin}[1]{%
  \begin{list}{}{%
    \setlength{\topsep}{0pt}%
    \setlength{\topmargin}{#1}%
    \setlength{\listparindent}{\parindent}%
    \setlength{\itemindent}{\parindent}%
    \setlength{\parsep}{\parskip}%
  }%
  \item[]}{\end{list}}

\parindent=0pt
\addtolength{\parskip}{0.5\baselineskip}

\title{Stabilisation of the LHS spectral sequence for algebraic groups}
\author{Alison E. Parker}
\address{School of Mathematics\\
University of Leeds\\
Leeds, LS2 9JT, UK\\
}
\email{a.e.parker@leeds.ac.uk {\text{\rm(Parker)}}}
\author{David I. Stewart}
\address{New College, Oxford\\ Oxford, UK}
\email{david.stewart@new.ox.ac.uk {\text{\rm(Stewart)}}}

\pagestyle{plain}
\begin{abstract}
In this note, we consider the Lyndon--Hochschild--Serre spectral sequence corresponding to the first Frobenius kernel of an algebraic group $G$ and computing the extensions between simple $G$-modules. We state and discuss a conjecture that $E_2=E_\infty$ and provide general conditions for low-dimensional terms on the $E_2$-page to be the same as the corresponding terms on the $E_{\infty}$-page, i.e. its abutment. 
\end{abstract}
\maketitle

\section{Introduction}Let $G$ be a reductive algebraic group over an
algebraically closed field $k$ of characteristic $p>0$. Let $\lambda$
and $\mu$ be two dominant weights for $G$. This paper concerns the
representation theory of $G$ and its first Frobenius kernel $G_1$; we refer to \cite{Jan03} for
notation. It is the purpose of this short note to state and provide some
evidence towards the following conjecture.

\begin{conjecture*}Suppose all $G_1$-injective hulls have the structure of $G$-modules, for instance if $p\geq 2h-2$. Then the
  Lyndon--Hochschild--Serre spectral
  sequence \[E_2^{ij}=\Ext^i_{G/G_1}(k,\Ext^j_{G_1}(L(\lambda),L(\mu)))\Rightarrow
  \Ext^{i+j}_G(L(\lambda),L(\mu))\tag{*}\] stabilises (i.e. reaches
  its abutment) at the $E_2$-page. That is, $E_2^{ij}\cong
  E_\infty^{ij}$ for all $i,j$.

Hence \[\Ext^n_G(L(\lambda),L(\mu))\cong
\bigoplus_{i+j=n}\Ext^i_{G/G_1}(k,\Ext^j_{G_1}(L(\lambda),L(\mu))).\]\end{conjecture*}
Note that it is an open conjecture of Humphreys and Verma that all $G_1$-injective hulls do indeed have the structure of $G$-modules, possibly making the first hypothesis trivially satisfied. 

Let us underline the fact that we are unaware of any occasion where any differential in the spectral sequence (*) is known to be non-zero---even after replacing $G$ with an arbitrary connected algebraic group and replacing $L(\lambda)$ and $L(\mu)$ by arbitrary $G$-modules. Showing that certain differentials in the
spectral sequence are zero has some history; we pick out a few
cases. For a large class of naturally occurring modules $V$ and $W$,
it was shown in \cite{Par07} that when $G=\SL_2$ the spectral sequence
does stabilise at the $E_2$-page. In particular the conjecture is
confirmed for the case $G=SL_2$, with no condition on $p$. It was
shown by Donkin in \cite{Don82} that the differentials
$d_{m,1}:E_2^{m,1}\to E_2^{m+2,0}$ are zero, also with no condition on
$p$. Some other special cases involving maps needed to compute second
cohomology were considered in work of McNinch \cite{McN02}, the second
author \cite{SteSL2,SteSL3}, and Ibraev \cite{Ibr11,Ibr12}.

Another case in which the conjecture is true is if $\lambda$ and $\mu$ are $p$-regular restricted weights, $p\geq 2h-2$ and $p$ is large enough that the Lusztig Character Formula holds. Then \cite[Theorem 5.3]{PS13} shows that the $G$-module $\Ext^n_{G_1}(L(\lambda),L(\mu))^{[-1]}$ has a good filtration for each $n$. Under these circumstances the spectral sequence moreover degenerates to a line; in particular the conjecture is true.

Note that the conjecture is not true if $G$ is replaced by an arbitrary group. See
\cite[\S6]{BF94}, \cite{Lea93} and \cite{Sie00} for examples of non-zero differentials.

The main theorem of this paper is a confirmation of the conjecture in
a generic sense. Here, the vanishing of differentials of degree much
lower than $p$ is guaranteed.
\begin{theorem*} Suppose $p\geq (r+1)(h-1)$. Then the differentials
  $d_n^{ij}$ in the spectral sequence (*) satisfying $i\le
  r-1$ and
  $n\geq 2$  or $j=0$ and $n \ge 2$ or $j=1$ and $n \ge 2$ are all zero.

In particular, \[\Ext^i_G(L(\mu),L(\lambda))\cong \bigoplus_{j=0}^i E_2^{i-j,j}\]
for $i \le r+1$.\end{theorem*}

We prove the above theorem by applying techniques from
\cite{Par07}. First, we show, in a proposition, that  part of a
minimal $G_1$-injective resolution has a compatible $G$-structure. We
then reconstruct the spectral sequence (*) in such a way that the
bottom-most complex in the double complex giving the $E_0$-page
contains this part of a minimal $G_1$-injective resolution. It follows
that many maps in the $E_0$-page are zero. Then some derived couple
arguments prove the theorem.

\section{Proposition and proof of the theorem}

In the proposition below, note that the case $r=0$ would be a special case of the Humphreys--Verma conjecture. (It is not known if the bound $p\geq 2h-2$ could be reduced to $p\geq h-1$ for $G_1$-injective hulls to lift to $G$-modules.)

\begin{prop*}Let $r\geq 1$ and let $\mu\in X_1$. Provided
  $p\geq (r+1)(h-1)$, there is a minimal $G_1$-resolution \[0\to
  L(\mu)\to I_0\to I_1\to \cdots \to I_r\to \cdots\] such that the
  sequence up to term $I_r$ has a
  $G$-structure.\end{prop*}

\begin{proof} We prove \emph{a fortiori} that there is such a sequence of $G$-modules with $I_r$ having weights $\lambda=\lambda_0+p\lambda_1$ with $\lambda_0\in X_1$, which satisfy $(\lambda_1,\alpha_0^\vee)\leq (r+1)(h-1)$.

First, let us treat the case $r=1$. Set $I_0=Q_1(\mu)$. The hypotheses imply that $p\geq 2h-2$; thus we know that $Q_1(\mu)$ has the structure of a $G$-module. The injection $L(\mu)\to Q_1(\mu)$ is then a map of $G$-modules. 

Let $M:=Q_1(\mu)/L(\mu)$. We may write $\soc_{G_1}M=\bigoplus_\nu L(\nu_0)\otimes M_\nu^F$ where $\nu_0\in X_1$ and $M_\nu$ is some $G$-module. Set $I_1=\bigoplus_{\nu}Q_1(\nu_0)\otimes M_\nu^F$. So $\soc_{G_1} I_1=\soc_{G_1} Q_1(\mu)/L(\mu)$. (It is worth noting that the condition on the weights here is enough to ensure that $\soc_{G_1}M=\soc_G M$ but we do not need this fact explicitly.) Thus $I_1$ is the $G_1$-injective hull of $M$, hence if there is a $G$-map $I_0\to I_1$, this will be part of a minimal resolution. It remains to show that there is indeed a map $I_0\to I_1$ of $G$-modules whose kernel is $L(\mu)$, i.e. a map $I_0/L(\mu)\to I_1$. Note that we do have a map $\soc_{G_1} M\to I_1$ by construction, so consider the exact sequence
\[\tag{*}\Hom_G(M,I_1)\to \Hom_G(\soc_{G_1} M,I_1) \to \Ext_G^1(M/\soc_{G_1} M,I_1).\]

If we could show that the third term in this sequence is zero then we would have that the first map were surjective, hence the $G$-map $\soc_{G_1} M\to I_1$ would lift to a map $M=I_0/L(\mu)\to I_1$ and we would be done.

Now $\Ext_G^1(M/\soc_{G_1} M,I_1)$ has a filtration by spaces $E=\Ext_G^1(M/\soc_{G_1} M,Q_1(\nu_0)\otimes L(\nu_1)^F)$ over certain weights $\nu=\nu_0+p\nu_1$. And $E$ can be computed via the 5-term exact sequence of the LHS spectral sequence, of which part is
\begin{align*}\Ext^1_{G/G_1}(k,\Hom_{G_1}&(M/\soc_{G_1} M,Q_1(\nu_0)\otimes L(\nu_1)^F))\to E\\&\to \Hom_{G/G_1}(k,\Ext^1_{G_1}(M/\soc_{G_1} M,Q_1(\nu_0)\otimes L(\nu_1)^F)).\end{align*}
Now the third term here is zero, as $Q_1(\nu_0)$ is injective for $G_1$, hence, to show $E=0$, it suffices to show that the first term is zero.

Now $I_0=Q_1(\mu)$, being a direct $G_1$-summand of $St_1\otimes L((p-1)\rho-\mu)$, has weights satisfying $\xi=\xi_0+p\xi_1$ with $(\xi_1,\alpha_0^\vee)\leq h-1$. Thus the composition factors of $I_0$ (hence of $M$) are of the form $L(\xi_0)\otimes L(\xi_1)^F$. In particular, we have that the weights $\nu_1$ satisfy $(\nu_1,\alpha_0^\vee)\leq h-1$. Thus $(\xi_1+\rho,\alpha_0^\vee),(\nu_1+\rho,\alpha_0^\vee)\leq 2h-2$ and our condition on $p$ implies that they are both in the closure of the lowest alcove, $\bar C_\Z$.  So let $L(\xi_0)\otimes L(\xi_1)^F$ be a composition factor of $M/\soc_{G_1} M$. We compute:
\begin{align*}\Ext^1_{G/G_1}(k,\Hom_{G_1}(L(\xi_0)\otimes L(\xi_1)^F,&Q_1(\nu_0)\otimes L(\nu_1)^F))\\&\cong\Ext^1_{G}(L(\xi_1),\Hom_{G_1}(L(\xi_0),Q_1(\nu_0))^{[-1]}\otimes L(\nu_1))\end{align*}
Now $\Hom_{G_1}(L(\xi_0),Q_1(\nu_0))$ is non-zero, thence equal to $k$, if an only $\xi_0=\nu_0$; in that case, the term on the right becomes $\Ext^1_{G}(L(\xi_1), L(\nu_1))$, and since $\xi_1,\nu_1\in \bar C_\Z$, this vanishes by the linkage principle. This concludes the proof in case $r=1$.

Now by induction we may assume that we have a sequence of
$G$-modules \[0\to I_0\to \cdots\to I_{r-2}\stackrel{\pi}{\to} I_{r-1},\] which is minimal as an injective $G_1$-resolution, such that the composition factors of $I_{r-1}$ have high weights $\lambda$ satisfying $\lambda=\lambda_0+p\lambda_1$ with $\lambda_0\in X_1$ and $(\lambda_1,\alpha_0^\vee)\leq r(h-1)$. We construct $I_r$ in a similar way to before: set $I_r=\bigoplus_{\nu}Q_1(\nu_0)\otimes M_\nu^F$, where the sum is over the $G$-composition factors of $\soc_{G_1}I_{r-1}/\pi I_{r-2}$, where $\nu_0\in X_1$ and a weight $\nu_1$ of $M_\nu$ satisfies $(\nu_1,\alpha_0^\vee)\leq r(h-1)$.  Thus a weight $\xi$ of $I_r$, say $\xi_0+p\xi_1$ with $\xi_0\in X_1$ satisfies $(\xi_1,\alpha_0^\vee)\leq (\nu_1,\alpha_0^\vee)+(\rho,\alpha_0^\vee)=(r+1)(h-1)$ as required. Note that $I_r$ is again a $G_1$-injective hull of $I_{r-1}/\im \pi$ so if we can show there is a $G$-module map $I_{r-1}\to I_r$ with kernel $\im\pi$, we will be done.

Of course, it is equivalent to produce a map from $M:=I_{r-1}/\im\pi$ to $I_r$. By construction we do have a map from $\soc_{G_1} M\to I_r$. Now the same argument as before shows that the third term in the sequence (*) (with $I_r$ replacing $I_1$) is zero. This completes the proof.

\end{proof}

\begin{proof}[Proof of the theorem]
We write $L(\mu) =
L(\mu_0)\otimes L(\mu_1)^F$ using Steinberg's tensor
product theorem where $\mu_0 \in X_1$ and $\mu_1 \in X^+$.
 Using the proposition we
have a $G$-resolution which is also a $G_1$-injective
resolution:
$$0 \to L(\mu_0) \to I_{0} \to I_{1} \to \cdots \to I_r\to\cdots,$$
where, up to $I_r$, the resolution is minimal for $G_1$.

We denote the differentials by $\delta_i:I_i \to I_{i+1}$ and the
kernels by 
$K_i : = \ker{\delta_i}$.
Dimension shifting gives us 
$\Ext^i_{G_1}(L(\lambda_0), L(\mu_0)) \cong 
\Ext^1_{G_1}(L(\lambda_0), K_{i-1})  $.
Minimality gives us for $\mu_1 \in X_1$ that 
$\Ext^i_{G_1}(L(\lambda_0), L(\mu_0)) \cong 
\Hom_{G_1}(L(\lambda_0), K_{i}) \cong
\Hom_{G_1}(L(\lambda_0), I_{i})$ for $i\leq r$.

We now have a $G$-resolution:
$$0 \to L(\mu) \to I_{0}\otimes L(\mu_1)^F
\stackrel{\partial_0}{\to} I_{1}
\otimes L(\mu_1)^{F}\stackrel{\partial_1}{\to} \cdots $$
where $\partial_i = \delta_i\otimes\id$,
as tensoring is exact. Also note that such a resolution stays
injective as a $G_1$-resolution as $L(\mu_1)^{F}$ is trivial as
  a $G_1$-module.

Now consider the $E_0$-page of the LHS spectral sequence that 
converges to  
$\Ext^*_G(L(\lambda),L(\mu))$
as constructed in \cite[\S2]{Par07}
$$E_0^{mn} = \Hom_{G/G_1}(k, \Hom_{G_1}(L(\lambda),I_n\otimes
L(\mu_1)^{F})\otimes J_m^F)
$$
where we have a $G$-injective resolution of the trivial module:
$$ 0 \to k \to J_0 \to J_1 \to \cdots $$
and this spectral sequence has $E_1$ and $E_2$ page 
\begin{align*}
E_1^{mn}&= \Hom_{G/G_1}(k, \Ext^n_{G_1}(L(\lambda),L(\mu))\otimes J_m^F)\\
E_2^{mn}&= H^m({G/G_1}, \Ext^n_{G_1}(L(\lambda),L(\mu))).
\end{align*}

Consider the 
induced maps $\partial^*_m$ in the
following complex, which has homology
$\Ext^*_{G_1}(L(\lambda),L(\mu))$:
$$
\Hom_{G_1}(L(\lambda), I_0\otimes L(\lambda_1)^{F}) \xrightarrow{\partial_0^*}
\Hom_{G_1}(L(\lambda), I_1\otimes L(\mu_1)^{F}) \xrightarrow{\partial_1^*}
\cdots.
$$
Now 
\begin{align*}
\Ext^m_{G_1}(L(\lambda),L(\mu))
&\cong
\Ext^m_{G_1}(L(\lambda_0),L(\mu_0)) \otimes L(\mu_1)^{F}\otimes
L(\lambda_1^*)^{F}\\
&\cong
\Hom_{G_1}(L(\lambda_0), I_{m})\otimes L(\mu_1)^{F}\otimes
L(\lambda_1^*)^{F}
\cong
\Hom_{G_1}(L(\lambda),  I_{m}\otimes L(\mu_1)^{F})
\end{align*}
for $m \le r$.
Thus all the differentials $\partial^*_m$ for $m \le r$ must be zero.

Now by \cite[\S3.2, \S3.4]{bensonII} we
know that the spectral sequence can be constructed using  derived couples.
We have 
$$D_0^{mn}= \bigoplus_{m+n=e+f, \ e \ge m}E^{ef}_0$$
$$E_1^{mn}= H(E^{mn}_0, d_0)$$
$$D_1^{mn}= H(E^{mn}_0 \oplus E_0^{m+1,n-1}\oplus\cdots, d_0+d_1)$$

We define the higher derived couples by taking the derived couple of
the previous one.
We have an exact diagram of doubly graded $k$-modules
$$
\xymatrix@R=20pt@C=10pt{
D_l \ar@{->}[rr]^{i_l}& &D_l \ar@{->}[dl]^{j_l}\\
& E_l \ar@{->}[ul]^{k_l} &
}
$$
The derived couple (for $l \ge 1$) is defined by
\begin{align*}
D_{l+1}^{mn}&= \im{i_l^{m+1, n-1}}\subseteq D_l^{mn}\fit &
E_{l+1}^{mn}&= H(E^{mn}_l, d_l)\fit \\
i_{l+1}^{mn}&=i_l^{mn}\Bigl\vert_{D_{l+1}} &
j_{l+1}^{mn}(i_l^{m+1,n-1}(x))&= j_l^{m+1,n-1}(x) + \im(d_l)\fit  \\
k_{l+1}^{mn}(z +\im(d_l))&= k_l^{mn}(z)\fit &
d_{l+1} &= j_{l+1} \circ k_{l+1}\fit 
\end{align*}
And the degrees of the maps $k$, $j$ and $d$ are:
$$
\deg(i_n) = (-1,1), \qquad 
\deg(j_n) = (n-1,n+1), \qquad 
\deg(k_n) = (1,0). 
$$

Now using \cite[Lemma 2.1]{Par07} we have that
$
d_0^{mn} =0
$ implies that $k_2^{m-1,n+1} =0$.
Thus since $d_0^{mn} =0$ for $m \le r$ we have $k_2^{mn}=0$ for $m
\le r-1$.
Thus all $k_l^{mn} =0$ for all $l \ge 2$ and $m \le r-1$.
As $d_l^{mn} =   j_l^{m+1,n}\circ k_l^{mn}$ we also get
$d_l^{mn} =0$ for all $l \ge 2$ and $m \le r-1$.

In other words, as all these differentials are zero on the $E_2$ page and
remain zero, the terms $E_2^{mn}$ with $m \le r-1$ must already be the
stable value. That is, $E_{\infty}^{mn} = E_2^{mn}$ for $m \le r-1$.

This easily gives us that
$$\Ext_G^i(L(\lambda), L(\mu)) = \bigoplus_{j=0}^i E^{i-j, j}_2$$
for $i < r$.
To get the result for $i =r$, we note that all the terms in the sum
$$\bigoplus_{j=0}^r E^{r-j, j}_{\infty}$$
stabilise at the $E_2$ page by the above, except, possibly the term
$E_\infty^{r0}$. But here clearly 
$E_\infty^{r0} = E_2^{r0}$ as all incoming differentials are zero by
the above, and the leaving differential $d_l^{r,0}$ is always zero as our spectral sequence is first quadrant.

We may similarly argue for $r+1$.
We consider
$$\bigoplus_{j=0}^{r+1} E^{r+1-j, j}_{\infty}.$$
As before all terms except possibly 
$E_\infty^{r,1}$ and $E_\infty^{r+1,0}$ stabilise at the $E_2$ page.
The same argument as in the previous case gives $E_{\infty}^{r+1,0} =
E_2^{r+1,0}$.

Now note that all incoming differentials to $E_l^{r,1}$ are zero for
$l \ge 2$ by the above. We also have that $d_l^{r,1} = 0$ for $l \ge
3$, again since the spectral sequence is first quadrant. So we need only check
that
$d_2^{r1}=0$, but this is true using \cite[Main Theorem]{Don82}. Thus we also get the result for $r+1$.
\end{proof}

\subsection*{Acknowledgements} The authors wish to thank Len Scott and Dan Nakano for comments on a previous version of this paper.

\bibliographystyle{amsalpha}
\bibliography{bib}

\end{document}